\documentclass{article}
\usepackage{amsmath,amssymb,amsthm}

\theoremstyle{plain}
\newtheorem{theorem}{Theorem}[section]
\newtheorem{definition}[theorem]{Definition}

\newtheorem{example}[theorem]{Example}

 \def\ivq{\in \! \vee \, {\rm q}}

\begin{document}
\title{\bf Soft  MTL-algebras  based on fuzzy sets}
 \normalsize

\author{  {Jianming Zhan}$^{a,*}$, {Wies{\l}aw A. Dudek}$^b$ \\ {\small
   $^{a }$ Department of Mathematics,
  Hubei Institute for Nationalities,}\\ {\small Enshi, Hubei Province,
   445000, P. R. China}\\
   {\small $^b$ Institute of Mathematics and Computer Science,
  Wroc{\l}aw University of Technology,}\\
  {\small Wybrze\.ze Wyspia\'nskiego 27, 50-370 Wroc{\l}aw, Poland}  \\}

\date{}\maketitle

\begin{flushleft}\rule[0.4cm]{12cm}{0.3pt}
\parbox[b]{12cm}{\small
\bf Abstract\rm \paragraph{ } In this paper, we deal with soft
MTL-algebras based on fuzzy sets. By means of $\in$-soft sets and
q-soft sets, some characterizations of (Boolean, G- and MV-)
filteristic soft MTL-algebras are investigated. Finally, we prove that a soft set is a
Boolean filteristic  soft MTL-algebra if and only if it is both a G-filteristic  soft
MTL-algebra and an MV-filteristic soft MTL-algebra.\\ \\


 }
\rule{12cm}{0.3pt}
\end{flushleft}
\footnote{* Corresponding author. Tel/Fax: 0086-718-8437732.
       \\ \it E-mail addresses:\rm \ \
zhanjianming@hotmail.com (J. Zhan),
dudek@im.pwr.wroc.pl (W.A. Dudek).}

\section{Introduction }

\paragraph{ }  To solve complicated problems in economics,
engineering, and environment, we cannot successfully use classical
methods because of various uncertainties typical for those problems.
There are three theories: theory of probability, theory of fuzzy
sets, and the interval mathematics which we can consider as
mathematical tools for dealing with uncertainties. But all these
theories have their own difficulties. Uncertainties cannot be
handled using traditional mathematical tools but may be dealt with
using a wide range of existing theories such as probability theory,
theory of (intuitionistic) fuzzy sets, theory of vague sets, theory
of interval mathematics, and theory of rough sets. However, all of
these theories have their own difficulties which have been  pointed
out in \cite{14}. Maji et al. \cite{12} and Molodtsov \cite{14} suggested that one
reason for these difficulties may be due to the inadequacy of the
parametrization tool of the theory. To overcome these difficulties,
Molodtsov \cite{14} introduced the concept of soft set as a new
mathematical tool for dealing with uncertainties that is free from
the difficulties that have troubled the usual theoretical
approaches. Molodtsov pointed out several directions for the
applications of soft sets. At present, research on the soft set
theory is progressing rapidly. Maji et al. \cite{13} described the
application of soft set theory to a decision making problem. They
also studied several operations on the theory of soft sets. The
algebraic structure of set theories dealing with uncertainties has
been studied by some authors. The most appropriate theory for
dealing with uncertainties is the theory of fuzzy sets developed by
Zadeh \cite{17, 18}. Jun \cite{5} applied the notion of soft sets by Molodtsov to the
theory of BCK/BCI-algebras, and introduced the notions of soft
BCK/BCI-algebras, and then investigated their basic properties
\cite{6}. Aktas et al. \cite{1} studied the basic concepts of soft set theory,
and compared soft sets to fuzzy and rough sets, providing some
examples to clarify their differences.

The interest in foundation of Fuzzy Logic has been rapidly
recently and several new algebras playing the role of the
structures of truth values has been introduced.
H$\mathrm{\acute{a}}$jek introduced the axiom system of basic
logic (BL) for fuzzy propositional logic and defined the class of
BL-algebras (see \cite{4}). The logic MTL, Monoidal t-norm based
logic was introduced by Esteva and Godo \cite{3}. This logic is
very interesting from many points of view. From the logic point of
view, it can be regarded as a weak system of Fuzzy Logic. In
connection with the logic MTL, Esteva and Godo \cite{3} introduced
a new algebra, called a MTL-algebra, and studied several basic
properties. In the same times independently were introduced in
\cite{21} weak-BL algebras as commutative weak-pseudo-BL algebras.
MTL-algebras and weak-BL algebras are the same algebras.

Based on the fuzzy set theory, Kim et al. in \cite{8} studied
the fuzzy structure of filters in MTL-algebras. As a continuation of
the paper \cite{8}, Jun et al. \cite{7} gave characterizations of fuzzy
filters in MTL-algebras and investigated further properties of fuzzy
filters in MTL-algebras. The other important results can be found in
\cite{16, 19}.

The idea of quasi-coincidence of a fuzzy point with a fuzzy set,
which was mentioned in \cite{15}, played a vital role to generate some
different types of fuzzy subsets. It is worth pointing out that
Bhakat and Das \cite{2} initiated the concepts of $(\alpha,\beta)$-fuzzy
subgroups by using the ``belongs to'' relation $(\in \, )$ and
``quasi-coincident with'' relation $({\rm q})$ between a fuzzy point
and a fuzzy subgroup, and introduced the concept of an $(\in,
\ivq)$-fuzzy subgroup. In fact, the $(\in,\ivq)$-fuzzy subgroup
is an important generalization of Rosenfeld's fuzzy subgroup. It is
now natural to investigate similar type of generalizations of the
existing fuzzy subsystems of other algebraic structures. With this
objective in view, Ma et al. \cite{9, 10, 11} discussed  some kind of generalized fuzzy
filters of MTL-algebras.

In this paper, we deal with soft MTL-algebras based on fuzzy sets.
In Section 2, we recall some basic defnitions of MTL-algebras. In
Section 3, we discuss the characterizations of filteristic soft
MTL-algebras. In Section 4, we divide into three parts. In
Subsection 4.1, we investigate some characterizations of Boolean
filteristic soft MTL-algebras. Some properties of MV- and G-filteristic
soft MTL-algebras are invetigated in Subsection 4.2 and 4.3,
respectively. Finally, we prove that a soft set is a Boolean filteristic soft
MTL-algebra if nd only if it is both a G-filteristic soft MTL-algebra and an
MV-filteristic soft MTL-algebra.

\section{Preliminaries }

\paragraph{ }By a {\it commutative, integral and bounded residuated lattice}
we shall mean a lattice $L=(L,\le,\wedge,\vee, \odot,\rightarrow,
0,1)$ containing the least element $0$ and the largest element $1\ne 0$, and
endowed with two binary operation $\odot$ (called product) and
$\rightarrow$ (called residuum) such that

$(1)$ \  $\odot$ is associative, commutative  and isotone,

$(2)$ \  $\forall x\in L$, $x\odot 1=x$,

$(3)$ \  the Galois correspondence holds, that is,

\hspace*{10mm}$\forall x,y,z\in L$,  $x\odot y\le z\Leftrightarrow x\le y\rightarrow
z.$

\medskip
In a commutative, integral and bounded  residuated lattice, the
following are true  (see \cite{16}):

$(1)$ \  $x\le y\Leftrightarrow x\rightarrow y=1$,

$(2)$ \ $ 0\rightarrow x=1, \ 1\rightarrow x=x, \ x\rightarrow(y\rightarrow
x)=1$,

$(3)$ \  $y\le (y\rightarrow x)\rightarrow x$,

$(4)$ \  $x\rightarrow (y\rightarrow z)=(x\odot y)\rightarrow
z=y\rightarrow (x\rightarrow z)$,

$(5)$ \   $x\rightarrow y\le (z\rightarrow x)\rightarrow (z\rightarrow
y), \ \   x\rightarrow y \le (y\rightarrow z)\rightarrow (x\rightarrow
z)$.

\paragraph{ }Based on  the H$\mathrm{\acute{a}}$jek's results \cite{4},  Axioms of MTL  and
Formulas which are provable in MTL,  Esteva  and  Godo  \cite{3} defined
the algebras,  so called MTL-algebras  corresponding to the
MTL-logic in the following way:

\paragraph{ } A   {\it MTL-algebra } is a commutative, integral and bounded  residuated lattice
$L=(L,\le,\wedge,\vee ,\odot,\rightarrow, 1)$   satisfying  the
pre-linearity   equation:

$$(x\rightarrow y)\vee (y\rightarrow  x)=1.$$

\paragraph{ }In a  MTL-algebra, the following are true:

\medskip
$(1)$ \  $ x\rightarrow  (y\vee z)=(x\rightarrow y)\vee  (x\rightarrow
z)$,

$(2)$ \  $ x\odot y\le  x\wedge y$,

$(3)$ \ $x'=x''', \ \  x\le x'', \ \ x'\odot x=0$,

$(4)$ \ if $x\vee x'=1$, then $x\wedge x'=0$,

where $x'=x\rightarrow 0.$

\medskip
Throughout this paper, $ L$ is  a MTL-algebra unless otherwise
specified.

We cite below some notations, definitions and basic results which
will be  needed in the sequel.

\paragraph{ }A non-empty  subset  $A$ of  $L$ is
called a  {\it filter}  of $L$  if it is closed undet the operation $\odot$ and for every $x\in A$, \ $x\le y$ implies $y\in A$.  It is easy to check that a non-empty subset $A$ of $L$ is a filter of $L$ if and only  if $1\in A$ and for all $x\in A$ from $x\rightarrow y\in A$ it follows $y\in A$.

\medskip
A  filter  $A$ of  $L$ is called:

$\bullet$ \ a {\it Boolean filter} if  $x\vee x'\in A$ for any $x\in L$,

$\bullet$  \ a {\it G-filter} if $ x\odot
x\rightarrow y\in A\Rightarrow x\rightarrow y\in A$ for any $x,y\in L$,

$\bullet$ \ an {\it MV-filter} if $ x\rightarrow y\in A\Rightarrow ((y\rightarrow
x)\rightarrow x)\rightarrow y\in A$ for any $x,y\in L$.

\paragraph{ }We now review some fuzzy logic concepts. A fuzzy set of  $L$ is a function $\mu: L\rightarrow [0,1]$.

Now, we recall some  the following concepts and results in \cite{7, 8, 20}.

\begin{definition}\label{D2.1}\rm  A fuzzy set $\mu$ of
$L$  is called  a {\it fuzzy filter} of $L$ if

$(F1)$ \ $\mu(x\odot y)\ge\min\{\mu(x),\mu(y)\}$ for all $x,y\in L$,

$(F2)$ \ it is order-preserving, that is, $x\le
y\Rightarrow \mu(x)\le \mu(y)$ for all $x,y\in L$.
\end{definition}

\begin{theorem}\label{T2.2} A fuzzy set $\mu$ of
$L$  is a {\it fuzzy filter} of $L$ if and only if

$(F3)$ \ $\mu(1) \ge\mu(x)$,

$(F4)$ \ $\mu(y)\ge\min\{\mu(x\rightarrow y),\mu(x)\}$

\noindent
is satisfied for all $x,y\in L$.
\end{theorem}

\begin{definition}\label{D2.3}\rm A fuzzy filter $\mu$ of $L$  is called
a {\it fuzzy Boolean filter }  of $L$  if
$\mu(x\vee  x')=\mu(1)$ holds for all $x\in L$.
\end{definition}

\begin{theorem}\label{T2.4} Let $\mu$ be a fuzzy filter of $L$,  then the following are equivalent:

\hspace*{2mm}$(i)$ \ $\mu$ is Boolean,

 \hspace*{1mm}$(ii)$ \ $\mu(x\rightarrow z)\ge\min\{\mu(x\rightarrow(z'\rightarrow y)),\mu(y\rightarrow z)\}$,

 $(iii)$ \ $\mu(x)\ge\mu((x\rightarrow y)\rightarrow x)$.
 \end{theorem}

\begin{definition}\label{D2.5}\rm A fuzzy  filter $\mu$ of $L$  is called

$\bullet$ \ a {\it fuzzy MV-filter }  if
 $\mu(x\rightarrow  y)\ge  \mu(((y\rightarrow x)\rightarrow
x)\rightarrow y)$,

$\bullet$ \ a {\it fuzzy G-filter } if
$\mu(x\odot x\rightarrow y)\ge  \mu(x\rightarrow y)$

\noindent for all $x,y\in L.$
\end{definition}

\section{Filteristic soft MTL-algebras}

 \paragraph { }  Molodtsov \cite{14} defined the soft set in the following
way: Let $U$ be an initial universe set and $E$ be a set of
parameters. Let $\mathcal{P}(U)$ denotes the power set of $U$ and
$A\subset E.$

A pair $(F,A)$ is called a {\it soft set} over $U,$ where $F$
is a mapping given by $F: A\to \mathcal{P}(U).$

In other words, a soft set over $U$ is a parameterized family of
subsets of the universe $U$. For $\varepsilon \in A,$
$F(\varepsilon)$ may be considered as the set of
$\varepsilon$-approximate elements of the soft set $(F,A).$

\begin{definition}\label{D3.1}\rm Let $(F,A)$ be a soft set over $L$. Then
$(F,A)$ is called a {\it filteristic soft MTL-algebra over $L$} if
$F(x)$ is a filter of $ L$ for all $x\in A$, for our convenience,
the empty set $\emptyset$ is regarded as a filter of $L$.
\end{definition}

\begin{example}\label{Ex3.2}\rm  Let $L=[0,1]$  and define  a product  $\odot$
and  a residuum  $\rightarrow$ on $L$  as follows:

\[x\odot y=\left\{\begin{array}{ll}x\wedge y & \mbox{ if $x+y>0.5,$  }\\
0& \mbox{ otherwise,}\end{array}\right.\ \
 x\rightarrow y=\left\{\begin{array}{ll}1 & \mbox{ if $x\le y$,  }\\
\max\{1-x,y\} & \mbox{ otherwise,}\end{array}\right.\]\\
 for all $x,y\in L$.  Then $L$ is an  MTL-algebra.

 Let $(F,A)$ be a soft set over $L$, where $A=(0,1]$ and $F:
 A\rightarrow \mathcal{P}(L)$ be a set-valued function defined by

$$  F(x)=\left\{\begin{array}{l l} L & \mbox{\ \ \ \ \
if\ \ \   }  0<x\le 0.5,\\ \{1\} & \mbox{\ \ \ \ \ if\ \ \ }
0.5<x\le 0.6,\\ \emptyset & \mbox{\ \ \ \ \ if\ \ \ } 0.8<x\le 1.
\end{array}\right.$$

Thus, $F(x)$ is a filter of $L$ for all $ x\in A$, and so $(F,A)$ is
a filteristic soft MTL-algebra over $L$.
\end{example}

For a fuzzy set $\mu$ in any MTL-algebra  $L$ and $A\subseteq
[0,1]$ we can consider two set-valued functions

 \[F: A\rightarrow \mathcal{P}(L), ~~t\mapsto \{x\in L\mid x_t\in
 \mu\}\]
and
 \[ F_{q}:A\rightarrow \mathcal{P}(L), ~~t\mapsto \{x\in L\mid x_t \, {\rm q} \, \mu\}.\]

Then $(F, A)$ and $(F_{q}, A)$  are called an {\it $\in$-soft
set} and {\it {\rm q}-soft set} over $L,$ respectively.

\begin{theorem}\label{T3.3} Let $\mu $ be a fuzzy set of $L$ and let
$(F,A)$ be an $\in$-soft set over $L$ with $A=(0,1]$. Then $(F,A)$ is a
filteristic soft MTL-algebra over $L$ if and only if $\mu$ is a
fuzzy filter of $L$.
\end{theorem}
\begin{proof} Let $\mu$ be a fuzzy filter of $L$ and $t\in A.$ If $x\in
F(t)$, then $x_t\in\mu$, and so $1_t\in\mu$, i.e., $1\in F(t)$. Let
$x,y\in L$ be such that $x, x\rightarrow y\in F(t)$. Then
$x_t\in\mu$ and $(x\rightarrow y)_t\in\mu$, and so
$y_{\min\{t,t\}}=y_t\in\mu$. Hence $y\in F(t)$. This proves that $(F,A)$ is a
filteristic soft MTL-algebra over $L$.

Conversely, assume that $(F,A)$ is a filteristic soft MTL-algebra
over $L$. If there exists $a\in L$ such that $\mu(1)<\mu(a)$, then
we can choose $t\in A$ such that $\mu(1)<t\le\mu(a)$. Thus,
$1_t\overline{\in}\mu$, i.e., $1\overline{\in} F(t)$. This is a
contradiction. Hence, $\mu(1)\ge \mu(x)$, for all $x\in L$. If there
exist $a,b\in L$ such that $\mu(b)<s\le \min\{\mu(a\rightarrow
b),\mu(a)\}$. Then $(a\rightarrow b)_s\in\mu$ and $a_s\in\mu$, but
$b_s\overline{\in}\mu$, that is, $a\rightarrow b\in F(s)$ and $a\in
F(s)$, but $b\overline{\in} F(s)$, contradiction, and so, $\mu(y)\ge
\min\{\mu(x\rightarrow y),\mu(x)\}$, for all $x,y\in L$. Therefore,
$\mu$ is a fuzzy filter of  $L$.
\end{proof}

\begin{theorem}\label{T3.4} Let $\mu$ be a fuzzy set of $L$ and $(F_{
q},A)$ a $q$-soft set over $L$ with $A=(0,1]$. Then the following are
equivalent:

\hspace*{1mm}$(i)$ \ $\mu$ is a fuzzy filter of $L$,

$(ii)$ \ $\forall t\in A$ each non-empty $F_{q}(t)$ is a filter of $L$.
\end{theorem}
\begin{proof}
Let $\mu$ be a fuzzy filter of $L$ and let $F_{q}(t)\ne\emptyset$ for any $t\in A$. If $1\overline{\in}F_{q}(t)$, then $1_t \overline{q}\mu$, and so $\mu(1)+t<1$. Then
$\mu(x)+t\le\mu(1)+t<1$ for all $x\in L$, and so $F_{q}(t)=\emptyset$, contradiction. Hence $1\in F_{q}(t)$.

Let $x,y\in L$ be such that $x\rightarrow y\in F_{q}(t)$ and $x\in
F_{q}(t)$. Then $(x\rightarrow y)_t q \mu$ and $x_t q\mu$, or
equivalently, $\mu(x\rightarrow y)+t>1$ and $\mu(x)+t>1$. Thus,
$$
\mu(y)+t\ge\min\{\mu(x\rightarrow y),\mu(x)\}+t=\min\{\mu(x\rightarrow y)+t, \mu(x)+t\}>1,
$$
and so $y_t q \mu,$ i.e., $y\in F_{q}(t)$. Hence $F_{q}(t)$
is a filter of $L$.

Conversely, assume that the condition $(ii)$ holds. If $\mu(1)<\mu(a)$
for some $a\in L$, then $\mu(1)+t\le 1<\mu(a)+t$ for some $t\in A.$
Thus, $a_t q\mu$, and so $F_{q}(t)\ne\emptyset$. Hence $1\in
F_{q}(t)$, and so $1_t q\mu$, i.e.,$ \mu(1)+t>1$, contradiction.
Hence $\mu(1)\ge\mu(x)$ for all $x\in L$.

If there exist $a,b\in L$ such that $\mu(b)<\min\{\mu(a\rightarrow
b),\mu(a)\}$. Then
$$
\mu(b)+s\le 1<\min\{\mu(a\rightarrow b),\mu(a)\}+s
$$
for some $s\in A$. Hence $(a\rightarrow b)_s q\mu$
and $a_s q\mu$, i.e., $a\rightarrow b\in F_{q}(s)$ and $a\in
F_{q}(s)$. Since $F_{q}(s)$ is a filter of $L$, we have
$b\in F_{q}(s)$, and so $b_s q\mu$, that is, $\mu(b)+s>1$,
contradiction. Hence $\mu(y)\ge \min\{\mu(x\rightarrow y),
\mu(x)\}$, for all $x,y\in L$. Therefore $\mu$ is a fuzzy filter of
$L.$
\end{proof}

\begin{definition}\label{D3.5}\rm  A fuzzy set $\mu$ of $L$ is  an
 {\it  $(\in,\ivq)$-fuzzy  filter }  of $L$ if for all $x,y\in L$ it satisfies:

 $(F5)$ \ $\mu(1)\ge\min\{\mu(x), 0.5\}$,

 $(F6)$ \  $\mu(y)\ge\min\{\mu(x\rightarrow y),\mu(x), 0.5\}$.
\end{definition}

\begin{theorem}\label{T3.6} Let $\mu$ be a fuzzy set of $L$ and
$(F,A)$ be an $\in$-soft set over $L$ with $A=(0,0.5]$. Then the
following are equivalent:

\hspace*{1mm}$(i)$ \ $\mu$ is an $(\in,\ivq)$-fuzzy filter of $L$,

$(ii)$ \ $(F,A)$ is a filteristic soft MTL-algebra over $ L$.
\end{theorem}
\begin{proof} Let $\mu$ be an $(\in,\ivq)$-fuzzy filter of $L$. For
any $t\in A$, we have $\mu(1)\ge\min\{\mu(x), 0.5\}$ for all $x\in
F(t)$ by Definition \ref{D3.5}. Hence
$\mu(1)\ge\min\{\mu(x),0.5\}\ge\min\{t,0.5\}=t$, which implies,
$1_t\in\mu$, and so  $1\in F(t)$. If $x\rightarrow y\in F(t)$ and
$x\in F(t)$, then $(x\rightarrow y)_t\in\mu$ and $x_t\in\mu$, that
is, $\mu(x\rightarrow y)\ge t$ and $\mu(x)\ge t$. Now, by $(F6)$, we have
$$
\mu(y)\ge\min\{\mu(x\rightarrow y),\mu(x),0.5\}\ge\min\{t,0.5\}=t,
$$
which implies, $y_t\in\mu$, and so $y\in F(t)$. Thus, $(F,A)$ is a
filteristic soft MTL-algebra  over $L$.

Now assume that the condition $(ii)$ holds. If there exists
$a\in L$ such that $\mu(1)<\min\{\mu(a),0.5\}$, then
$\mu(1)<t\le\min\{\mu(a), 0.5\}$ for some $t\in A$. It follows that
$1_t\overline{\in} \mu$, i.e., $1\overline{\in} F(t)$,
contradiction. Hence $\mu(1)\ge\min\{\mu(x),0.5\}$ for all $x\in L$.
If there exist $a,b\in L$ such that $\mu(b)<\min\{\mu(a\rightarrow
b),\mu(a),0.5\}$, then taking
$t=\frac{1}{2}(\mu(b)+\min\{\mu(a\rightarrow b),\mu(a),0.5\})$, we
have $t\in A$ and
$$
\mu(b)<t<\min\{\mu(a\rightarrow b),\mu(a),0.5\},
$$
which implies, $a\rightarrow b\in F(t), a\in F(t)$, but
$b\overline{\in} F(t)$, contradiction. It follows from Definition
\ref{D3.5}  that $\mu$ is an $(\in,\ivq)$-fuzzy filter of $L$.
\end{proof}

\begin{definition}\label{D3.7}\rm \cite{9} A fuzzy set $\mu$  of $L$
is called an {\it $(\overline{\in},\overline{\in} \vee
\overline{q})$-fuzzy filter} of $L$ if and only if for $x,y\in L$ it satisfies:

$(F7)$ \ $\max\{\mu(1),0.5\}  \ge  \mu(x)$,

$(F8)$ \  $\max\{\mu(y),0.5\}\ge\min\{\mu(x\rightarrow y), \mu(x)\}$.
\end{definition}

\begin{theorem}\label{T3.8} Let $\mu$ be a fuzzy set of $L$ and
$(F,A)$ be an $\in$-soft set over $L$ with $A=(0.5,1]$. Then the
following are equivalent:

\hspace*{1mm}$(i)$ \ $\mu$ is an $(\overline{\in},\overline{\in} \vee\overline{
q})$-fuzzy filter of $L$,

$(ii)$ \ $(F,A)$ is a filteristic soft MTL-algebra over $L$.
\end{theorem}
\begin{proof} Let $\mu$ be an $(\overline{\in},\overline{\in} \vee\overline{ q})$-fuzzy
filter of $L$. For any $t\in A$, by Definition \ref{D3.7}, we have
 $\mu(x)\le\max\{\mu(1), 0.5\}$ for all $x\in F(t)$. Thus,
$t\le\mu(x)\le\max\{\mu(1),0.5\}=\mu(1)$, which implies $1_t\in\mu$,
i.e., $1\in F(t)$.

Let $x,y\in L$ be such that $x\rightarrow y\in
F(t)$ and $x\in F(t)$, then $(x\rightarrow y)_t\in\mu$ and
$x_t\in\mu$, i.e., $\mu(x\rightarrow y)\ge t$ and $\mu(x)\ge t$. It
follows from Definition \ref{D3.7} that
$$
t\le\min\{\mu(x\rightarrow
y),\mu(x)\}\le\max\{\mu(y),0.5\}=\mu(y),
$$
which implies, $y_t\in\mu$, i.e., $y\in F(t)$. Hence $F(t) $ is a
filter of $L$ for all $t\in A$, and so $(F,A)$ is a filteristic soft
MTL-algebra  over $L.$

Now, assume that $(F,A)$ is a filteristic soft MTL-algebra
over $L$. If there exists $a\in L$ such that $\mu(a)\ge
\max\{\mu(1),0.5\}$, then $\mu(a)\ge t>\max\{\mu(1),0.5\} $ for some
$t\in A$, and so $\mu(1)<t$.  Thus, $1\overline{\in} F(a)$.  Contradiction.
Hence $\mu(x)\le \max\{\mu(1),0.5\}$ for all $x\in L.$ If there
exist $a,b\in L$ such that $\min\{\mu(a\rightarrow b),\mu(a)\}\ge
t>\max\{\mu(b),0.5\}$ for some $t\in A$, then $(a\rightarrow b)_t$ and $a_t$ are in $\mu$. But $b_t\overline{\in}\mu$, therefore $a\rightarrow b,a\in F(t)$. This is a contradiction since
$b\overline{\in} F(t)$. It follows from Definition \ref{D3.7} that $\mu$ is an
$(\overline{\in},\overline{\in} \vee\overline{ q})$-fuzzy filter of
$L.$
\end{proof}

Next, we give the following two important results by $q$-soft
sets.

\begin{theorem}\label{T3.9} Let $\mu$ be a fuzzy set of $L$ and
$(F_{q},A)$ be a q-soft set over $L$ with $A=(0,0.5]$. Then
$(F_{q},A)$ is a filteristic soft MTL-algebra over $ L$ if and
only if $\mu$ is an $(\overline{\in},\overline{\in} \vee
\overline{q})$-fuzzy filter of $L$.
\end{theorem}
\begin{proof} Let $(F_{q},A)$ be a filteristic soft MTL-algebra  over
$ L$, then $F_{q}(t)$ is a filter of $L$ for all $t\in A$. If
$\max\{\mu(1),0.5\}<\mu(a)$ for some $a\in L$, then
$\max\{\mu(1),0.5\}+t\le 1<\mu(a)+t$ for some $t\in A.$ Thus, $1_t\overline{q}\mu$,
which is impossible. Hence $\max\{\mu(1),0.5\}\ge\mu(x)$ for all $x\in L$.

If there exist $a,b\in L$ such that
$\max\{\mu(b),0.5\}<\min\{\mu(a\rightarrow b),\mu(a)\}$. Then
$\max\{\mu(b),0.5\}+s\le 1<\min\{\mu(a\rightarrow b),\mu(a)\}+s$ for
some $s\in A$. Hence $(a\rightarrow b)_s q\mu$ and $a_s q\mu$, i.e.,
$a\rightarrow b\in F_{q}(s)$ and $a\in F_{q}(s)$. Since
$F_{q}(s)$ is a filter of $L$, we have $b\in F_{q}(s)$, and
so $b_s q\mu$, that is, $\mu(b)+s>1$, contradiction. Hence
$\max\{\mu(y),0.5\}\ge \min\{\mu(x\rightarrow y), \mu(x)\}$, for all
$x,y\in L$. Therefore $\mu$ is an  $(\overline{\in},\overline{\in}
\vee \overline{q})$-fuzzy filter of $L.$

Conversely, let $\mu$ be an  $(\overline{\in},\overline{\in} \vee
\overline{q})$-fuzzy filter of $L$. For any $t\in A$. By Definition \ref{D3.7}, we have $\mu(x)\le\max\{\mu(1),0.5\}$ for all $x\in F_{\rm
q}(t)$,  and so  $\max\{\mu(1),0.5\}+t\ge\mu(x)+t>1$. Hence
$\mu(1)+t>1$, that is, $1\in F_{q}(t)$.  Let $x,y\in L$ be such
that $x\rightarrow y\in F_{q}(t)$ and $x\in F_{q}(t)$. Then
$(x\rightarrow y)_t q \mu$ and $x_t q\mu$, or equivalently,
$\mu(x\rightarrow y)+t>1$ and $\mu(x)+t>1$. By Definition \ref{D3.7}, we have
$$
\max\{\mu(y),0.5\}+t\ge\min\{\mu(x\rightarrow y),\mu(x)\}+t=\min\{\mu(x\rightarrow y)+t, \mu(x)+t\}>1,
$$
and so $y_t q\mu,$ i.e., $y\in F_{q}(t).$ Hence $F_{q}(t)$
is a filter of $L$, and so $(F_{q},A)$ is a filteristic soft
MTL-algebra  over $L$.
  \end{proof}

\begin{theorem}\label{T3.10} Let $\mu$ be a fuzzy set of $L$ and
$(F_{q},A)$ be a q-soft set over $L$ with $A=(0.5,1]$. Then
$(F_{q},A)$ is a filteristic soft MTL-algebra over $L$ if and only
if $\mu$ is an  $(\in,\ivq)$-fuzzy filter of $L$.
\end{theorem}
\begin{proof} Let $(F_{q},A)$ be a filteristic soft MTL-algebra  over
$L$. Then $F_{q}(t)$ is a filter of $L$ for all $t\in A$.  If
$\mu(1)<\min\{\mu(a),0.5\}$ for some $a\in L$, then $\mu(1)+t\le
1<\min\{\mu(a),0.5\}+t$ for some $t\in A.$ Thus,
$1_t\overline{q}\mu$, contradiction. Hence
$\mu(1)\ge\min\{\mu(x),0.5\}$ for all $x\in L$.

If there exist $a,b\in L$ such that $\mu(b)<\min\{\mu(a\rightarrow
b),\mu(a),0.5\}$. Then $\mu(b)+s\le 1<\min\{\mu(a\rightarrow
b),\mu(a),0.5\}+s$ for some $s\in A$. Hence $(a\rightarrow b)_s
q\mu$ and $a_s q\mu$, i.e., $a\rightarrow b\in F_{q}(s)$ and
$a\in F_{q}(s)$. Since $F_{q}(s)$ is a filter of $L$, we
have $b\in F_{q}(s)$, and so $b_s q\mu$, that is, $\mu(b)+s>1$,
contradiction. Hence $\mu(y)\ge \min\{\mu(x\rightarrow y), \mu(x),
0.5\}$, for all $x,y\in L$. Therefore $\mu$ is an $(\in,\ivq)$-fuzzy
filter of $L.$

Conversely, let $\mu$ be an $(\in,\ivq)$-fuzzy filter of $L$. By Definition \ref{D3.5}, we have
$\mu(1)\ge\min\{\mu(x),0.5\}$ for all $x\in F_{q}(t)$,  and so
$\mu(1)+t\ge\min\{\mu(x),0.5\}+t=\min\{\mu(x)+t,0.5+t\}>1$. Hence
$\mu(1)+t>1$, that is, $1\in F_{q}(t)$.

Now, let $x,y\in L$ be such
that $x\rightarrow y\in F_{q}(t)$ and $x\in F_{q}(t)$. Then
$(x\rightarrow y)_t q \mu$ and $x_t q\mu$, or equivalently,
$\mu(x\rightarrow y)+t>1$ and $\mu(x)+t>1$. Thus
\[
\begin{array}{rl}
\mu(y)+t&\ge\min\{\mu(x\rightarrow y),\mu(x),0.5\}+t\\[3pt]
&=\min\{\mu(x\rightarrow y)+t, \mu(x)+t, 0.5+t\}>1,
\end{array}
\]
and so $y_t q \mu,$ i.e., $y\in F_{q}(t).$ Hence $F_{q}(t)$
is a filter of $L$, and consequently, $(F_{q},A)$ is a filteristic soft
MTL-algebra  over $L$.
\end{proof}

\begin{definition}\label{D3.11}\rm For $0<\alpha<\beta\leq 1$ a fuzzy  set $\mu$ of $L$ is
called a {\it fuzzy filter with thresholds $(\alpha,\beta]$} if for all $x,y\in L$:

\hspace*{1.5mm}$(F9)$ \  $\max\{\mu(1), \alpha\}\ge \min\{\mu(x) , \beta \}$,

$(F10)$ \ $\max\{\mu(y), \alpha\}\ge \min\{\mu(x\rightarrow y), \mu(x),
\beta\}$.
\end{definition}

\begin{theorem}\label{T3.12}
Let $\mu$ be a fuzzy set of $L$. Then an $\in$-soft set $(F,A)$ over $L$ with
$A=(\alpha,\beta]\subset (0,1]$ is a filteristic soft  MTL-algebra over $L$ if and only if $\mu$ is a fuzzy filter with thresholds $(\alpha,\beta]$.
\end{theorem}
\begin{proof} Let $(F,A)$ be a filteristic soft MTL-algebra as in Theorem. If there exists $a\in L$ such that
$\max\{\mu(1),\alpha\}<\min\{\mu(a),\beta\}$, then
$\max\{\mu(1),\alpha\}<t\le\min\{\mu(a),\beta\}$  for some $t\in
(\alpha,\beta]$. Thus $1\overline{\in}F(t)$ which is a contradiction. If there  exist $a,b\in L$ such that
$\max\{\mu(b),\alpha\}<t\le\min\{\mu(a\rightarrow
b),\mu(a),\beta\}.$ Hence $(a\rightarrow b)_t\in\mu, a_t\in\mu$. But
$b_t\overline{\in}\mu$, therefore $a\rightarrow b\in F(t), a\in F(t).$ This also is a contradiction since
$b\overline{\in} F(t)$. Consequently, $\mu$ is a fuzzy filter with thresholds $(\alpha,\beta]$ of $L.$

On the other hand, if $\mu$ is a fuzzy filter with thresholds $(\alpha,\beta]$, then, by $(F9)$, we have
$\max\{\mu(1),\alpha\}\ge \min\{\mu(x),\beta\}$ for all $x\in F(t)$.
Thus, $\max\{\mu(1),\alpha\}\ge \min\{\mu(x),\beta\}\ge
\min\{t,\beta\}=t>\alpha$, which implies, $\mu(1)\ge t$, i.e.,
$1_t\in \mu$. Hence $1\in F(t)$. Let $x,y\in L$ be such that
$x\rightarrow y \in F(t)$ and $x\in F(t)$. Thus, $(x\rightarrow y)_t
\in \mu$ and $x_t \in \mu$, i.e., $\mu(x\rightarrow y)\ge t$ and
$\mu(x)\ge t.$ By $(F10)$, we have
$\max\{\mu(y),\alpha\}\ge \min\{\mu(x\rightarrow
y),\mu(x),\beta\}\ge \min\{t,\beta\}=t>\alpha$, and so $\mu(y) \ge
t$, i.e., $y_t\in \mu$, and so $y\in F(t)$. Therefore, $(F,A)$ is a
filteristic soft MTL-algebra over $L$.
\end{proof}

\section{Boolean (MV-, G-) filteristic soft MTL-algebras}

\paragraph { } In this section divided in three parts we describe some types of generalized fuzzy filters of
MTL-algebras introduced in \cite{10}. In the first part we
characterize Boolean filteristic soft MTL-algebras; in the second
-- G-filtersistic soft MTL-algebras; in third -- MV-filtersistic
soft MTL-algebras which are natural generalizations of Boolean
filters, G-filters and MV-filters, respectively. Finally,
 we describe relationship between these soft MTL-algebras.

\subsection{Boolean  filteristic soft MTL-algebras}
\paragraph{ } We start with the following definition.

\paragraph{Definition 4.1.1.}\label{D4.1.1}\rm A soft set $(F,A)$ over $L$ is called a {\it Boolean filteristic soft MTL-algebra over $L$}
if $F(x)$ is a Boolean filter of $L$ for all $x\in A$.
The empty set is treatment as a Boolean filter.\\

\paragraph{Example 4.1.2.}\label{Ex4.1.2} Consider the set
$L=\{0,a,b,1\}$ with two operations defined by the following
tables:

\begin{center}
\begin{tabular}{c|cccc}
          $\odot$ & 0 & $a$ & $b$ & 1  \\ \hline
          0 & 0 & 0 & 0 & 0  \\
          $a$ & 0 & $a$ & $a$ & $a$ \\
          $b$ & 0 & $a$ & $a$ & $b$  \\
          1 & 0 & $a$ & $b$ & 1
   \end{tabular}
\ \ \ \ \ \ \ \ \
\begin{tabular}{c|cccc}
          $\rightarrow$ & 0 & $a$ & $b$ & 1  \\ \hline
          0 & 1 & 1 & 1 & 1 \\
          $a$ & $0$ & 1 & 1 & 1  \\
          $b$ & 0 & $b$ & 1 & 1  \\
          1 & 0 & $a$ & $b$ & 1
\end{tabular}
  \end{center}
Then $(L,\wedge,\vee,\odot,\rightarrow, 0,1)$, where $\wedge$ and $\vee$ are $\min$ and
$\max$ operations, respectively, is an MTL-algebra.

Let $(F,A)$ be a soft set over $L$, where $A=(0,1]$ and $F:
A\rightarrow \mathcal{P}(L)$ be a set-valued function defined by

$$  F(x)=\left\{\begin{array}{c l} \{0,a,b,1\} & \mbox{\ \ \ \ \
if\ \ \   }  0<x\le 0.4,\\ \{1,a,b\} & \mbox{\ \ \ \ \ if\ \ \ }
0.4<t\le 0.8,\\ \emptyset & \mbox{\ \ \ \ \ if\ \ \ } 0.8<t\le 1.
\end{array}\right.
$$
Thus, $F(x)$ is a Boolean filter of $L$ for all $ x\in A$, and so
$(F,A)$ is a Boolean filteristic soft MTL-algebra over $L$.

\paragraph{ }The following proposition is obvious.

\paragraph{Proposition 4.1.3.}\label{P4.1.3} {\it A Boolean
filteristic MTL-algebra is a filteristic MTL-algebra. }

\paragraph{Theorem 4.1.4.}\label{T4.1.4} {\it
Let $\mu$ be a fuzzy set of $L$. Then an $\in$-soft set $(F,A)$
over $L$ with $A=(0,1]$ is a Boolean filteristic soft MTL-algebra
over $L$ if and only if $\mu$ is a fuzzy Boolean filter of $L$.}
\begin{proof}
Let $(F,A)$ an $\in$-soft set $(F,A)$ over $L$ with $A=(0,1]$. If
it is a Boolean filteristic soft MTL-algebra over $L$, then, by
Proposition 4.1.3, it is a filteristic soft MTL-algebra over $L$,
and so $\mu$ is a fuzzy filter of $L$ (Theorem \ref{T3.3}). If
there exist $a,b,c\in L$ such that $\mu(a\rightarrow c)<s\le
\min\{\mu(a\rightarrow(c'\rightarrow b)),\mu(b\rightarrow c)\}$
for some $s\in A$. Then $(a\rightarrow b)_s\mu$ and $a_s\in\mu$,
but $b_s\overline{\in}\mu$, that is, $a\rightarrow (c'\rightarrow
b)\in F(s)$ and $b\rightarrow c\in F(s)$. Thus $a\rightarrow
c\overline{\in} F(s)$, which is a contradiction. Therefore, $\mu$
is a fuzzy Boolean filter of $L$.

Conversely, if $\mu$ is a fuzzy Boolean filter of $L$, then it is
also a fuzzy filter of $L$ and, by Theorem \ref{T3.3}, $(F,A)$ is
a filteristic soft MTL-algebra over $L$.  Let $x,y,z\in L$ be such
that $x\rightarrow (z'\rightarrow y), y\rightarrow z\in F(t)$.
Then $(x\rightarrow (z'\rightarrow y))_t\in\mu$ and $(y\rightarrow
z)_t\in\mu$. Hence, by Theorem \ref{T2.4}, we obtain
$\mu(x\rightarrow z)\ge\min\{\mu(x\rightarrow(z'\rightarrow
y)),\mu(y\rightarrow z)\}\ge t$, and so   $x\rightarrow z\in
F(t)$. This proves (Theorem \ref{T2.4}) that $(F,A)$ is a Boolean
filteristic soft MTL-algebra over $L$.
\end{proof}

\paragraph{Theorem 4.1.5.}\label{T4.1.5}
{\it Let $\mu$ be a fuzzy set of $L$. If $(F_{q},A)$, where
$A=(0,1]$, is a q-soft set over $L$, then $\mu$ is a fuzzy Boolean
filter if and only if each non-empty $F_{q}(t)$ is a Boolean
filter.}

\begin{proof} Let $\mu$ be a fuzzy  Boolean  filter of
$L$. Then, by Theorem \ref{T3.4}, $F_{q}(t)$ is a filter of $L$.
Let $x,y,z\in L$ be such that $x\rightarrow(z'\rightarrow y)\in
F_{q}(t)$ and $y\rightarrow z\in F_{q}(t)$. Then
$(x\rightarrow(z'\rightarrow y))_t q \mu$ and $(y\rightarrow z)_t
q\mu$, or equivalently, $\mu(x\rightarrow (z'\rightarrow y))+t>1$
and $\mu(y\rightarrow z)+t>1$. Since $\mu$ is a fuzzy  Boolean of
$L$, we have
\[\arraycolsep=.5mm
\begin{array}{rl}
\mu(x\rightarrow z)+t&\ge\min\{\mu(x\rightarrow (z'\rightarrow
y)),\mu(y\rightarrow z)\}+t\\[2pt]
&=\min\{\mu(x\rightarrow (z'\rightarrow z))+t, \mu(y\rightarrow
z)+t\}>1, \end{array}
 \] and so $(x\rightarrow z)_t q$, i.e.,
$x\rightarrow z\in F_{\rm q}(t)$. This proves (Theorem \ref{T2.4})
that $F_{q}(t)$ is a Boolean filter of $L$.

Conversely, assume that each non-empty $F_{q}(t)$ is a Boolean
filter of $L$. Then $\mu$ is a fuzzy filter of $L$ by Theorem
\ref{T3.4}. If there exist $a,b,c\in L$ such that
$\mu(a\rightarrow c)<\min\{\mu(a\rightarrow (c'\rightarrow
b)),\mu(b\rightarrow c)\}.$ Then $\mu(a\rightarrow c)+s\le
1<\min\{\mu(a\rightarrow (c'\rightarrow b)),\mu(b\rightarrow
c)\}+s$ for some $s\in A$. Hence $(a\rightarrow (c'\rightarrow
b))_s q\mu$ and $(b\rightarrow c)_s q\mu$, but $(a\rightarrow c)_t
\overline{q}\mu$,  i.e., $a\rightarrow (c'\rightarrow b)\in F_{\rm
q}(s)$ and $b\rightarrow c\in F_{q}(s)$, but $a\rightarrow
c\overline{\in} F_{q}(t),$ contradiction. Therefore $\mu$ is a
fuzzy Boolean filter of $L$.
\end{proof}

\paragraph{Definition 4.1.6.}\label{D4.1.6} An $(\in,\in\vee q)$-fuzzy filter $\mu$ of $L$ is called
an {\it $(\in,\in\vee q)$-fuzzy Boolean filter} of $L$ if
$$
\mu(x\rightarrow  z)\ge \min\{\mu(x\rightarrow(z'\rightarrow y)),
\mu(y\rightarrow z),0.5 \}
$$
holds for all $x\in L.$

\paragraph{Theorem 4.1.7.}\label{T4.1.7} {\it Let $\mu$ be a fuzzy set of $L$. Then
an $\in$-soft set $(F,A)$ over $L$ with $A=(0,0.5]$ is a Boolean
filteristic soft MTL-algebra if and only if $\mu$ is an
$(\in,\ivq)$-fuzzy Boolean filter of $L$.}

\begin{proof} Let an $\in$-soft set $(F,A)$, where $A=(0,0.5]$, be a Boolean
filteristic soft MTL-algebra. If there exist $a,b,c\in L$ such
that
$$
\mu(a\rightarrow c)<\min\{\mu(a\rightarrow (c'\rightarrow
b)),\mu(b\rightarrow c),0.5\},
$$
then for
$$
t=\frac{1}{2}(\mu(a\rightarrow c)+ \min\{\mu(a\rightarrow
(c'\rightarrow b),\mu(b\rightarrow c),0.5\})
$$
we have $t\in A$ and
$$
\mu(a\rightarrow c)<t<\min\{\mu(a\rightarrow (c'\rightarrow
b)),\mu(b\rightarrow c),0.5\},
$$
which implies $ a\rightarrow (c'\rightarrow b)\in F(t)$ and
$b\rightarrow c\in F(t)$. This is a contradiction since
$a\rightarrow c\overline{\in} F(t)$. So, $\mu$ is an
$(\in,\ivq)$-fuzzy Boolean filter of $L$.

Conversely, if $\mu$ is an $(\in,\ivq)$-fuzzy Boolean filter of
$L$, then, by Theorem \ref{T3.6}, $(F,A)$ is a filteristic soft
MTL-algebra. Moreover, if $x,y,z\in L$ be such that
$x\rightarrow(z'\rightarrow y)\in F(t)$ and $y\rightarrow z\in
F(t)$ for some $t\in A$, then $\mu(x\rightarrow (z'\rightarrow
y))\ge t$ and $\mu(y\rightarrow z)\ge t$. Thus,
$$
\mu(x\rightarrow z)\ge\min\{\mu(x\rightarrow (z'\rightarrow
y)),\mu(y\rightarrow z),0.5\}\ge\min\{t,0.5\}=t,
$$
which implies $(x\rightarrow z)_t\in\mu$, and so $x\rightarrow
z\in F(t)$. Hence, $(F,A)$ is a Boolean  filteristic soft
MTL-algebra over $L$.
\end{proof}

\paragraph{Definition 4.1.8.}\label{D4.1.8}   An $(\overline{\in},\overline{\in} \vee
\overline{q})$-fuzzy filter of $L$ is called an {\it
$(\overline{\in},\overline{\in} \vee \overline{q})$-fuzzy Boolean
filter} of $L$ if it satisfies:
$$
\max\{\mu(x\rightarrow z),0.5\}\ge
\min\{\mu(x\rightarrow(z'\rightarrow y)),\mu(y\rightarrow z)\}
$$
for all $x\in L$.

\paragraph{Theorem 4.1.9.}\label{T4.1.9} {\it For a fuzzy set $\mu$ of $L$ and
an $\in$-soft set $(F,A)$ over $L$ with $A=(0.5,1]$ the following
conditions are equivalent:

$(i)$ \ $\mu$ is an $(\overline{\in},\overline{\in} \vee\overline{
q})$-fuzzy Boolean filter of $L$,

$(ii)$ \ $(F,A)$ is a Boolean  filteristic soft MTL-algebra of
$L$.}
\begin{proof}  Let $\mu$ be an $(\overline{\in},\overline{\in} \vee\overline{ q})$-fuzzy
Boolean  filter of $L$, then $\mu$ is also an
$(\overline{\in},\overline{\in}\vee\overline{q})$-fuzzy filter of
$L$ and, by Theorem \ref{T3.8}, $(F,A)$ is a filteristic soft
MTL-algebra. Let $x,y,z\in L$ be such that $x\rightarrow
(z'\rightarrow y)\in F(t)$ and $y\rightarrow z\in F(t)$ for some
$t\in A$, then $(x\rightarrow (z'\rightarrow y))_t\in\mu$ and
$(y\rightarrow z)_t\in\mu$, i.e., $\mu(x\rightarrow (z'\rightarrow
y))\ge t$ and $\mu(y\rightarrow z)\ge t$. Thus,
$$
t\le\min\{\mu(x\rightarrow (z'\rightarrow y)),\mu(y\rightarrow
z)\}\le\max\{\mu(x\rightarrow z),0.5\}=\mu(x\rightarrow z),
$$
which implies $(x\rightarrow z)_t\in\mu$, i.e., $x\rightarrow z\in
F(t)$. Hence $F(t)$ is a Boolean filter of $L$, and so $(F,A)$ is
a Boolean filteristic soft MTL-algebra over $L$.

Conversely, assume that $(F,A)$ is a Boolean filteristic soft
MTL-algebra over $L$. If for some $t\in A$ there exist $a,b,c\in
L$ such that
$$
\min\{\mu(a\rightarrow (c'\rightarrow b)),\mu(b\rightarrow c)\}\ge
t>\max\{\mu(a\rightarrow c),0.5\},
$$
then $(a\rightarrow (c'\rightarrow c))_t\in \mu$ and
$(b\rightarrow c)_t\in \mu$ but $(a\rightarrow
c)_t\overline{\in}\mu$. This means that $a\rightarrow
(c'\rightarrow b)\in F(t)$, $b\rightarrow c\in F(t)$, but
$a\rightarrow c\overline{\in} F(t)$, which is a contradiction.
Therefore, $\mu$ is an $(\overline{\in},\overline{\in}
\vee\overline{q})$-fuzzy Boolean filter of $L$.
\end{proof}

\paragraph{ }Now, we give the following two important characterizations of Boolean filterstic $q$-soft
sets.

\paragraph{Theorem 4.1.10.}\label{T4.1.10} {\it Let $\mu$ be a fuzzy set of $L$. Then
a $q$-soft set $(F_{q},A)$ over $L$ with $A=(0,0.5]$ is a Boolean
filteristic soft MTL-algebra over $ L$ if and only if $\mu$ is an
$(\overline{\in},\overline{\in} \vee \overline{q})$-fuzzy Boolean
filter of $L$.}

\medskip
The proof is similar to the proof of Theorem \ref{T3.9}.

\paragraph{Theorem 4.1.11.}\label{T4.1.11}{\it Let $\mu$ be a fuzzy set of $L$. Then
a $q$-soft set $(F_{q},A)$ of $L$ with $A=(0.5,1]$ is a Boolean
filteristic soft MTL-algebra over $ L$ if and only if $\mu$ is an
$(\in,\ivq)$-fuzzy Boolean filter of $L$.}

\medskip
The proof is similar to the proof of Theorem \ref{T3.10}.

\paragraph{} As a consequence of Theorems 3.6, 4.1.5 and 4.1.7 we obtain

\paragraph{Theorem 4.1.12.}\label{T4.1.12} {\it For a fuzzy set $\mu$ of $L$ and an $\in$-soft set $(F,A)$
over $L$ with $A=(\alpha,\beta]$, where $0<\alpha<\beta\leq 1$,
the following conditions are equivalent:

$(i)$ \ $\mu$ is a fuzzy Boolean filter with thresholds
$(\alpha,\beta]$ of $L$,

$(ii)$ \ $(F,A)$ is a Boolean filteristic soft MTL-algebra over
$L$.}

\subsection{MV-filteristic soft MTL-algebras}

\paragraph{ } In this subsection, we characterize MV-filteristic soft MTL-algebras by fuzzy MV-filters.

\paragraph{Definition 4.2.1.} A soft set $(F,A)$ over $L$ is called an {\it MV-filteristic soft MTL-algebra over
$L$} if $F(x)$ is an MV-filter of $L$ for all $x\in A$. The empty
set is treatment as an MV-filter of $L$.

\paragraph{Example 4.2.2.}\label{E42.2}  Let $L=\{0,a,b,1\}$ be a chain with operations defined by the following
two tables:

\begin{center}
\begin{tabular}{c|cccc}
          $\odot$ & 0 & $a$ & $b$ & 1  \\ \hline
          0 & 0 & 0 & 0 & 0  \\
          $a$ & 0 & $0$ & $0$ & $a$ \\
          $b$ & 0 & $0$ & $a$ & $b$  \\
          1 & 0 & $a$ & $b$ & 1
   \end{tabular}
\ \ \ \ \ \ \ \ \
\begin{tabular}{c|cccc}
          $\rightarrow$ & 0 & $a$ & $b$ & 1  \\ \hline
          0 & 1 & 1 & 1 & 1 \\
          $a$ & $b$ & 1 & 1 & 1  \\
          $b$ & $a$ & $b$ & 1 & 1  \\
          1 & 0 & $a$ & $b$ & 1
\end{tabular}
  \end{center}
Then $(L,\wedge,\vee,\odot,\rightarrow, 0,1)$, where $\wedge$ and
$\vee$ are $\min$ and $\max$ operations, respectively, is an
MTL-algebra.

Let $(F,A)$ be a soft set over $L$, where $A=(0,1]$ and $F:
A\rightarrow \mathcal{P}(L)$ be a set-valued function defined by

$$  F(x)=\left\{\begin{array}{c l} \{0,a,b,1\} & \mbox{\ \
if\ \ \   }  0<x\le 0.4,\\ \{1\} & \mbox{\ \  if\ \ \ } 0.4<t\le
0.8,\\ \emptyset & \mbox{\ \ if\ \ \ } 0.8<t\le 1.
\end{array}\right.
$$
Thus, $F(x)$ is an MV-filter of $L$ for all $ x\in A$, and so
$(F,A)$ is an MV-filteristic soft MTL-algebra over $L$.

\paragraph{ }From the above definitions, we can get the following:

\paragraph{Proposition 4.2.3.} {\it Every MV-filteristic MTL-algebra
 is a filteristic MTL-algebra, but the converse may not be true. }

\paragraph{Theorem 4.2.4.} {\it Let $\mu$ be a fuzzy set of $L$. Then
an $\in$-soft set $(F,A)$  over $L$ with $A=(0,1]$ is an
MV-filteristic soft MTL-algebra over $L$ if and only if $\mu$ is a
fuzzy MV-filter of $L$.}

\medskip

The proof is similar to the proof of Theorem 4.1.4.

\paragraph{Theorem 4.2.5.} {\it For a fuzzy set $\mu$ of $L$ and
a $q$-soft set $(F_{q},A)$ over $L$ with $A=(0,1]$ the following
conditions are equivalent:

$(i)$ \ $\mu$ is a fuzzy MV-filter of $L$,

$(ii)$ \ each non-empty $F_{q}(t)$ is an MV-filter of $L$.}

\medskip

The proof is similar to the proof of Theorem 4.1.5.

\paragraph{Definition 4.2.6.} An $(\in,\in\vee
q)$-fuzzy filter $\mu$ of $L$ is called an {\it $(\in,\in\vee
q)$-fuzzy MV-filter } of $L$ if
\[
\mu(((y\rightarrow x)\rightarrow x)\rightarrow y)\ge
\mathrm{min}\{\mu(x\rightarrow y), 0.5\}
\]
is satisfied for all $x,y\in L$.

\paragraph{Theorem 4.2.7.} {\it For a fuzzy set $\mu$ of $L$ and
an $\in$-soft set $(F,A)$ over $L$ with $A=(0,0.5]$ the following
conditions are equivalent:

$(i)$ \ $\mu$ is an $(\in,\ivq)$-fuzzy MV-filter of $L$,

$(ii)$ \ $(F,A)$ is an MV-filteristic soft MTL-algebra over $L$.}

\medskip

The proof is similar to the proof of Theorem 4.1.7.

\paragraph{Definition 4.2.8.} An $(\overline{\in},\overline{\in} \vee
\overline{q})$-fuzzy filter of $L$ is called an {\it
$(\overline{\in},\overline{\in} \vee \overline{q})$-fuzzy
 MV-filter} of $L$ if
\[
\max\{\mu(((y\rightarrow x)\rightarrow x)\rightarrow y),0.5\}\ge
\min\{\mu(x\rightarrow y)\}
\]
is satisfied for all $x,y\in L$.

\paragraph{Theorem 4.2.9.} {\it For a fuzzy set $\mu$ of $L$ and
an $\in$-soft set $(F,A)$ over $L$ with $A=(0.5,1]$ the following
conditions are equivalent:

$(i)$ \ $\mu$ is an $(\overline{\in},\overline{\in} \vee\overline{
q})$-fuzzy MV-filter of $L$,

$(ii)$ \ $(F,A)$ is an MV-filteristic soft MTL-algebra over $L.$}

\medskip
The proof is similar to the proof of Theorem 4.1.9.

\paragraph{Theorem 4.2.10.} {\it Let $\mu$ be a fuzzy set of $L$. Then
a $q$-soft set $(F_{q},A)$ over $L$ with $A=(0,0.5]$ is an
MV-filteristic soft soft MTL-algebra if and only if $\mu$ is an
$(\overline{\in},\overline{\in}\vee\overline{q})$-fuzzy MV-filter
of $L$.}

\medskip
The proof is similar to the proof of Theorem 3.9.

\paragraph{Theorem 4.2.11.} {\it Let $\mu$ be a fuzzy set of $L$. Then
a $q$-soft set $(F_{q},A)$ over $L$ with $A=(0.5,1]$ is a
filteristic soft MV-algebra over $L$ if and only if $\mu$ is an
$(\in,\ivq)$-fuzzy MV-filter of $L$.}

\medskip
The proof is similar to the proof of Theorem 3.10.

\medskip
As a consequence of Theorems 3.7, 4.2.5 and 4.2.7 we obtain

\paragraph{Theorem 4.2.12.} {\it For a fuzzy set $\mu$ of $L$ and an $\in$-soft set $(F,A)$
over $L$ with $A=(\alpha,\beta]$, where $0<\alpha<\beta\leq 1$,
the following conditions are equivalent:

$(i)$ \ $\mu$ is a fuzzy  MV-filter with thresholds
$(\alpha,\beta]$ of $L$,

$(ii)$ \ $(F,A)$ is an MV-filteristic soft MTL-algebra over $L$.}

\medskip
From Theorems 4.1.6, 4.1.7, 4.1.8, 4.2.6, 4.2.7, 4.2.9 and Theorem
3.20 in \cite{7} it follows

\paragraph{Theorem 4.2.13.} {\it Let $\mu$ be a fuzzy set of $L$. If an
$\in$-soft set $(F,A)$ over $L$ with $A=(\alpha,\beta]\subset
(0,1]$ is a Boolean filteristic soft MTL-algebra, then it also is
an MV-filteristic soft MTL-algebra, but the converse may not be
true.}

\subsection{G-filteristic soft MTL-algebras}

\paragraph{ } Now, we describe filteristic soft MTL-algebras connected with G-filters.

\paragraph{Definition 4.3.1.} A soft set $(F,A)$ over $L$ is called a {\it G-filteristic soft MTL-algebra over $L$}
if $F(x)$ is a G-filter of $ L$ for all $x\in A$. The empty set is
regarded as a G-filter of $L$.

\medskip

Since G-filter is a filter every G-filteristic MTL-algebra is a
filteristic MTL-algebra, but the converse is not be true in
general.

\paragraph{Example 4.3.2.} Consider on $L=[0,1]$ two operations $\odot$
and $\rightarrow$ defined by the following tables:
$$
\begin{tabular}{c | c c c c
c c} $\odot$& $0$ & $a$ &  $b$ &  $c$ &  $d$& $1$\\\hline
$0$ &  $0$ &  $0$ & $0$ & $0$ & $0$ & $0$\\
$a$ &  $0$ &  $a$ & $c$ & $c$ & $0$ & $a$\\
$b$ &  $0$ &  $c$ & $b$ & $c$ & $d$ & $b$\\
$c$ &  $0$ &  $c$ & $c$ & $c$ & $0$ & $c$\\
$d$ &  $0$ &  $0$ & $d$ & $0$& $0$ & $d$\\
$1$ &  $0$ &  $a$ & $b$ & $c$ & $d$ & $1$
\end{tabular}
\hspace{1.5cm}
\begin{tabular}{c | c c c c c c} $\rightarrow$ & $0$ &  $a$ &
$b$ &  $c$  & $d$ &  $1$ \\\hline
$0$ &  $1$ &  $1$ & $1$ & $1$ & $1$ & $1$\\
$a$ &  $d$ &  $1$ & $b$ & $b$ & $d$ & $1$\\
$b$ &  $0$ &  $a$ & $1$ & $a$ & $d$ & $1$\\
$c$ &  $d$ &  $1$ & $1$ & $1$ & $d$ & $1$\\
$d$ &  $a$ &  $1$ & $1$ & $1$ & $1$ & $1$\\
$1$ &  $0$ &  $a$ & $b$ & $c$ & $d$ & $1$\end{tabular}
$$
Then $L(\wedge,\vee,\odot,\rightarrow ,0,1)$ is a MTL-algebra.

Let $(F,A)$ be a soft set over $L$, where $A=(0,1]$ and $F:
A\rightarrow \mathcal{P}(L)$ be a set-valued function defined by
$$
F(x)=\left\{\begin{array}{c l} \{0,a,b,c,d,1\} & \mbox{\ \ \ if\ \
\   }  0<x\le 0.4,\\ \{1,a\} & \mbox{\ \ \  if\ \ \ } 0.4<t\le
0.8,\\ \emptyset & \mbox{\ \ \  if\ \ \ } 0.8<t\le 1.
\end{array}\right.
$$

Thus, $F(x)$ is a G-filter of $L$ for all $ x\in A$, and so $(F,A)$
is a G-filteristic soft MTL-algebra over $L$.

\medskip
In a similar way as Theorem 4.1.4 we can prove

\paragraph{Theorem 4.3.3.} {\it Let $\mu$ be a fuzzy set of $L$. Then an $\in$-soft set
$(F,A)$ over $L$ with $A=(0,1]$ is a G-filteristic soft
MTL-algebra over $L$ if and only if $\mu$ is a fuzzy G-filter of
$L$.}

\paragraph{Theorem 4.3.4.}
{\it Let $\mu$ be a fuzzy set of $L$. If $(F_{q},A)$, where
$A=(0,1]$, is a q-soft set over $L$, then $\mu$ is a fuzzy
G-filter if and only if each non-empty $F_{q}(t)$ is a G-filter.}

\begin{proof} The proof is similar to the proof of Theorem
4.1.5.
\end{proof}

\paragraph{Definition 4.3.5.} An $(\in,\in\vee q)$-fuzzy filter $\mu$ of $L$ is called
an {\it $(\in,\in\vee q)$-fuzzy G-filter} if
\[
\mu(x\rightarrow y)\ge \min\{\mu(x\odot x\rightarrow y), 0.5\}
\]
holds for all $x,y\in L$.

\paragraph{Theorem 4.3.6.} {\it Let $\mu$ be a fuzzy set of $L$. Then
an $\in$-soft set $(F,A)$ over $L$ with $A=(0,0.5]$ is a
G-filteristic soft MTL-algebra if and only if $\mu$ is an
$(\in,\ivq)$-fuzzy G-filter of $L$.}
 \begin{proof} The proof is similar to the proof of Theorem 4.1.7.
\end{proof}

\paragraph{Definition 4.3.7.} An $(\overline{\in},\overline{\in} \vee
\overline{q})$-fuzzy filter of $L$ is called an {\it
$(\overline{\in},\overline{\in} \vee \overline{q})$-fuzzy
G-filter} of $L$ if
\[
\max\{\mu(x\rightarrow y),0.5\}\ge\min\{\mu(x\odot x\rightarrow
y)\}
\]
holds for all $x,y\in L$.

\paragraph{Theorem 4.3.8.} {\it For a fuzzy set $\mu$ of $L$ and
an $\in$-soft set $(F,A)$ over $L$ with $A=(0.5,1]$ the following
conditions are equivalent:

$(i)$ \ $\mu$ is an $(\overline{\in},\overline{\in} \vee\overline{
q})$-fuzzy G-filter of $L$,

$(ii)$ \ $(F,A)$ is a G-filteristic soft MTL-algebra over $L.$}

\begin{proof} The proof is analogous  to the proof of Theorem
4.1.9.
\end{proof}

\paragraph{ }Also the proofs of the following two theorems are very similar to the proofs
of Theorems 3.9 and 3.10, respectively.

\paragraph{Theorem 4.3.9.} {\it Let $\mu$ be a fuzzy set of $L$. Then
a $q$-soft set $(F_{q},A)$ over $L$ with $A=(0,0.5]$ is a
G-filteristic soft MTL-algebra if and only if $\mu$ is an
$(\overline{\in},\overline{\in}\vee\overline{q})$-fuzzy G-filter.}

\paragraph{Theorem 4.3.10.}{\it Let $\mu$ be a fuzzy set of $L$. Then a $q$-soft set $(F_{q},A)$ over $L$
with $A=(0.5,1]$ is a G-filteristic soft MTL-algebra if and only
if $\mu$ is an $(\in,\ivq)$-fuzzy G-filter of $L$.}

\paragraph{} As a consequence of Theorems 3.7, 4.3.4 and 4.3.6 we obtain

\paragraph{Theorem 4.3.11.}
{\it For a fuzzy set $\mu$ of $L$ and an $\in$-soft set $(F,A)$
over $L$ with $A=(\alpha,\beta]$, where $0<\alpha<\beta\leq 1$,
the following conditions are equivalent:

$(i)$ \ $\mu$ is a fuzzy G-filter with thresholds
$(\alpha,\beta]$,

$(ii)$ \ $(F,A)$ is a G-filteristic soft MTL-algebra over $L$.}

\paragraph{} Finally, we give the relationship between the
filteristic soft MTL-algebras described above.

\paragraph{Theorem 4.3.12.} {\it Let $\mu$ be a fuzzy set of $L$. If an
$\in$-soft set $(F,A)$ over $L$ with $A=(\alpha,\beta]\subset
(0,1]$ is a Boolean filteristic soft MTL-algebra, then it also is
a G-filteristic soft MTL-algebra, but the converse may not be
true.}

\begin{proof} It is a consequence of Theorems 4.1.6,
4.1.7, 4.1.8, 4.3.5, 4.3.6, 4.3.8 and Theorem 4.5 in
\cite{20}.
\end{proof}

\paragraph{Theorem 4.3.13.}
{\it Let $\mu$ be a fuzzy set of $L$. Then an $\in$-soft set
$(F,A)$ over $L$ with $A=(\alpha,\beta]\subset (0,1]$ is a Boolean
filteristic soft MTL-algebra if and only if it is both an
MV-filteristic soft MTL-algebra and a G-filteristic soft
MTL-algebra.}

\begin{proof} It is a consequence of Theorems 4.2.13,
4.3.12 and Theorem 4.5 in \cite{20}.
\end{proof}

\subsection*{5. Conclusion}

\paragraph{ } In this paper, we apply fuzzy and soft set theory
to MTL-algebras.  We hope that the research along this direction can
be continued, and in fact, some results in this paper have already
constituted a platform for further discussion concerning the future
development of soft MTL-algebras and other algebraic structure.

In our future study of  MTL-algebras, may be the following topics
should be considered:

(1) To describe the soft MTL-algebras based on rough sets;

(2) To discuss the relations between soft MTL-algebras based on
fuzzy sets and rough sets;

(3) To consider the soft implication-based fuzzy  filters in
MTL-algebras.

\subsection*{Acknowledgements }

\paragraph{ } The research   is partially
supported by    the National Natural Science Foundation of China
(60875034);  the Natural Science Foundation of Education Committee
of Hubei Province, China (D20092901; Q20092907;  D20082903;
B200529001) and  the Natural Science Foundation of Hubei Province,
China (2008CDB341).

\end{document}